\documentclass[a4paper,11pt]{article}
\usepackage{a4wide,url}
\usepackage{amssymb,amsmath,amsthm,longtable}
\usepackage[T1]{fontenc}
\usepackage{color}
\usepackage{graphicx}
\usepackage{floatflt}
\usepackage{paralist}
\usepackage[font = small,labelfont=bf,format=hang,justification=raggedright]{caption}
\usepackage{verbatim}
\usepackage{lipsum}
\usepackage{hyperref}
\usepackage{todonotes}

\usepackage{caption}
\usepackage{subcaption}
\newcommand{\N}{\mathbb{N}}

\newcommand{\R}{\mathbb{R}}

\newcommand{\with}{, \quad}
\newcommand{\for}{, \quad \text{ for }}
\newcommand{\fa}{, \quad \text{ for all }}

\newtheorem{theorem}{Theorem}[section]

\newtheorem{definition}[theorem]{Definition}

\newtheorem{lemma}[theorem]{Lemma}

\DeclareMathOperator*{\argmin}{arg\,min}

\title{Detection of Geometric Structures by an Optimally Subsampled Shearlet System in Noisy Digital Images}

\author{Philipp Petersen}

\date{}

\begin{document}
\maketitle

%\listoftodos

\begin{abstract}
We provide a statistical analysis of the ability of digitized continuous shearlet systems to detect objects embedded in white noise. We analyze the possibility to subsample the shearlet system and obtain a system of significantly reduced cardinality that can still yield statistically optimal detection results.
\end{abstract}

%\maketitle

%\tableofcontents

%%%%%%%%%%%%%%%%%%%%%%%
%Introduction
%%%%%%%%%%%%%%%%%%%%%%%

\section{Introduction}\label{sec:introduction}
Detection and classification of singularities in images is an important task in imaging applications. It has been shown in \cite{KLCharacterizationOfEdges2009, GLLEdgeAnalysis2009} that the continuous shearlet transform can be applied to identify location and position of jump singularities of an image. These works have been extended and improved in \cite{KP2014AnaSing} to include compactly supported shearlets, which have optimal spatial localization and can therefore handle singularities better.

All of the above results use a continuous shearlet system i.e. a system constructed from a generator $\psi\in L^2(\R^2)$ together with elements
\begin{align}\label{eq:shearletSystem}
\psi_{a,s,t}(x) = \psi(A_a^{-1} S_s^{-1} (x - t))\with a \in (0,1],\ s \in [-1,1],\ t \in \R^2, 
\end{align}
where 
$$A_a = \begin{pmatrix} a &0\\ 0&\sqrt{a}\\ \end{pmatrix}\text{ and }S_s = \begin{pmatrix} 1&s\\0&1\\ \end{pmatrix}.$$

For applications we need to sample the parameters $a,s,t$ and also restrict the functions $\psi_{a,s,t}$ to a digital lattice or an $n\times n$ pixel image. 

In this work we will examine this digitization from a detection point of view, in the sense, that we aim to describe a suitable digital shearlet system to optimally detect geometrical structures embedded in noise. This is inspired by the work of Arias-Castro, Donoho and Huo \cite{AriDH2005DetectOfGeoObj}, which describes how to construct statistical tests to test for some geometric structures such as lines, intervals or rectangles. There it is assumed that $\xi$ is a discrete signal of length $n\in \N$ with certain characteristics such as jumps and a white noise signal $z(i)$, where $i$ is an element of some discrete index set $I_n$ and $z(i) \sim \mathcal{N}(0,1)$. In the one dimensional situation we choose $I_n: = \{1, \dots,  n\}$. Similarly in higher dimensions we let $I_n := \{i = (i_1, i_2, \dots, i_k): 1\leq i_l \leq n, l= 1,\dots, k\}$. The measured signal is now given by
\begin{align} \label{eq:signalInNoise}
 x(i) = \xi(i) + z(i)\with i \in I_n.
\end{align}
In \cite{AriDH2005DetectOfGeoObj} this model was used to detect discrete lines, rectangles, circles and intervals embedded in white noise. We will continue on this path and examine the optimal detection of structures connected to discontinuities and continuous shearlet systems. Our main result, Theorem \ref{thm:mainTheorem}, states, that careful subsampling of a continuous shearlet system does not influence its abilities to detect certain geometric properties. This subsampling yields a reduction of the size of the system from $O(n^4)$ to $O(n^2 \log(n))$ for an $n\times n$ pixel image.
The manuscript is structured as follows. We will describe the methodology of using statistical tests for the detection of geometrical structures in Subsection \ref{sec:Setup}. Then we will proceed to construct a test to find jump sdiscontinuities in 1d signals in Subsection \ref{sec:1DJumpSing}. From these methods we will extract a general framework in Subsection \ref{sec:generalFramework}. This will then be applied to the continuous shearlet system in Subsection \ref{sec:2DJump}.
\section{Detection of Jump Singularities}
\subsection{Problem Setup}\label{sec:Setup}
We will begin with the detection of jump discontinuities. From the point of view of detection this geometrical structure is closely related to that of intervals and therefore we will first briefly recall the detection of intervals. Using the model \eqref{eq:signalInNoise} with $\xi = A_n \xi_{[a,b]}(i)$, where $\xi_{[a,b]}(i) = \sqrt{b-a}\chi_{[a,b]}$,  $I_n = 1, \dots n$, the authors of \cite{AriDH2005DetectOfGeoObj} showed, that a test can reliably the presence of an interval in an asymptotically stable way, that will be introduced below, from the signal
\begin{align}\label{eq:intervalDet}
x(i) = A_n \xi_{[a,b]}(i) + z(i)\with 1\leq i \leq n,
\end{align}
if $A_n$ lies above a critical threshold. On the other hand if $A_n$ lies below this threshold, no test can reliably detect the interval. The phrase of asymptotically reliable detection needs some mathematical rigor and hence, we repeat the definitions of \textit{asymptotically powerful} and \textit{asymptotically powerless} tests.
\begin{definition}\cite{AriDH2005DetectOfGeoObj}
 In a sequence of testing problems $(H_{0,n})$ vs. $(H_{1,n})$, we say that a sequence of tests $(T_n)$ is \emph{asymptotically powerful} if 
 \begin{align*}
  P_{H_{0,n}}\left\{ T_n \text{ rejects } H_0 \right \} + P_{H_{1,n}}\left\{ T_n \text{ accepts } H_0 \right \} \to 0\with n \to \infty 
 \end{align*}
and the sequence is \emph{asymptotically powerless} if
\begin{align*}
   P_{H_{0,n}}\left\{ T_n \text{ rejects } H_0 \right \} + P_{H_{1,n}}\left\{ T_n \text{ accepts } H_0 \right \} \to 1\with n \to \infty.
\end{align*}
\end{definition}

\subsection{Detecting 1D Jump Singularities}\label{sec:1DJumpSing}
We aim to construct a series of tests that are asymptotically powerful, to detect jump singularities in noisy data. This means, that we obtain data of the form
\begin{align}
x(i) = \xi_J(i) + z(i)\with 1\leq i\leq n, 
\end{align}
where $\xi_J$ exhibits some sort of jump, and $z(i) \sim \mathcal{N}(0,1)$. In the spirit of \cite{AriDH2005DetectOfGeoObj} one must first pose the question under which conditions any test can reliably detect the presence of a jump in the signal.

First of all we need to determine what a jump singularity should be. At this point we may want to consider the prototype of a jump singularity, given by objects of the form $\psi(i) = \chi_{[1, s]}(i), i\in I_n$ which model a jump at the position $s$ from $1$ to $0$. These types of jumps are not very challenging, in fact the detection of these is already covered by the possibility to detect intervals as in \eqref{eq:intervalDet}. Another approach could be to model jumps simply as a point where the function value differs strongly from that of the neighboring points. Taking pairwise differences leads towards the problem of finding one elevated mean within white noise, i.e. \emph{the needle in a haystack problem}, \cite{NeedleInAHaystack1, NeedleInAHaystack2}. In order to be reliably detectable this would require the jumps to be quite high and in fact higher 
than we might want to. In fact, a jump could also be detectable, if it is small but in a neighborhood apart from the jump the function values of the signal change very little. In other words, if before and after the jump we have an almost constant signal, then we would also expect to be able to detect smaller jumps. Indeed, a reasonable model for a jump singularity is that a function has a large Haar wavelet coefficient. This will now be our model.

Let us examine the continuous Haar wavelet transform on $1\leq i \leq n$. Let for $1\leq a \leq n/2$, $1\leq t\leq n$,  
\begin{align}\label{eq:defOfHaar}
\psi_{a,t}(i) = \frac{1}{\sqrt{2a}}(\chi_{[0,a]}(i-t+a) - \chi_{[a,2a]}(i-t+a))\with 1\leq i \leq n, 
\end{align}
where the expression $i-t+a$ is to be understood modulo $n$. We denote by 
$$\mathcal{W}:= \left \{ \psi_{a,t}, 1\leq t \leq n, 1\leq a\leq n/2 \right \}$$ 
the continuous Haar wavelet system. The cardinality of the continuous Haar wavelet system is $n^2/2$. We choose the continuous Haar wavelet system as a model for jumps in the signal i.e. a signal has a jump of size $A$ at $t$, if $|\left \langle \xi, \psi_{a,t} \right \rangle| = A$ for some $1\leq a\leq n/2$.

Of course our jump singularity model is not the only possibility. One could certainly also work with unbalanced jumps, which means that we only need a smooth part with small variation on one side of the jump. Should one take the continuous Haar wavelet transform of white noise, i.e. look at the values $\left \langle \psi_{a,t}, z \right \rangle$, a simple calculation yields that we obtain $n^2/2$ values that are normal distributed. If the values $\left \langle \psi_{a,t}, x \right \rangle$ where to be independent, the needle in a haystack problem suggest, that detection of a jump is only reliably possible if the jump size is larger than $\sqrt{2\log(n^2)}$. However, the probability variables $\left \langle \psi_{a,t}, x \right \rangle$ are not independent and we will show that the optimal threshold for reliable detection is $\sup_{a, t} |\left \langle \psi_{a,t}, x \right \rangle| > \sqrt{2\log(n)}$. This means, that if for some $\eta>0$, $\sup_{a, t} |\left \langle \psi_{a,t}, x \right \rangle| \leq (1-\eta)
\sqrt{2\log(n)}$ 
there does not exist an asymptotically powerful test, but if $\sup_{a, t}| \left \langle \psi_{a,t}, x \right \rangle |\geq (1+\eta)\sqrt{2\log(n)}$ we can construct an asymptotically powerful test.

The key to show that these lower thresholds can be achieved is to subsample the scaling parameter of the continuous Haar wavelet system. For $\epsilon >0$, we choose values $a_i = 2^{\epsilon k}, k= 1, \dots \lceil\frac{1}{\epsilon}(\log(n)-1)\rceil$ and define
\begin{align*}
 \mathcal{W}_\epsilon : = \left \{ \psi_{a_k,t}, k= 1, \dots \lceil\frac{1}{\epsilon}(\log(n)-1)\rceil, 1\leq t \leq n \right \}.
\end{align*}
We use the metric: $\delta(\xi_1, \xi_2) := (1-\left \langle \xi_1, \xi_2\right \rangle)^{\frac{1}{2}}$ to describe the difference between the elements of $\mathcal{W}_\epsilon$ and the continuous Haar wavelet system $\mathcal{W}$. The following Lemma shows, that the subsampled system is in some sense dense with respect to the metric $\delta$.
\begin{lemma}\label{lem:EpsilonNet}
For $\epsilon >0$ the subsampled Haar wavelet system $\mathcal{W}_\epsilon$ is an $\epsilon$-net for the continuous Haar wavelet system $\mathcal{W}$, i.e. 
$$\max_{a, t} \min_{k, s} \delta(\psi_{a_k, s}, \psi_{a, t})< \epsilon.$$
\end{lemma}
\begin{proof}
For given $a\leq n/2$, there exists $a_k>0$ such that either $a_k \leq a \leq 2^{\epsilon/2} a_k$ or $2^{\epsilon/2} a_k\leq a \leq a_{k+1}$. Then
$$\delta(\psi_{a,t}, \psi_{a_k,t})^2 = 1- \left \langle \psi_{a,t}, \psi_{a_k,t} \right \rangle,$$
where
$$\left \langle \psi_{a,t}, \psi_{a_k,t} \right \rangle = \left \langle \psi_{a_k,t}, \psi_{a_k,t} \right \rangle + \left \langle \psi_{a_k,t} - \psi_{a,t}, \psi_{a_k,t} \right \rangle = 1 + \left \langle \psi_{a_k,t} - \psi_{a,t}, \psi_{a_k,t} \right \rangle.$$
Furthermore
\begin{align*}
 \left \langle \psi_{a_k,t} - \psi_{a,t}, \psi_{a_k,t} \right \rangle \leq \|\psi_{a_k,t} - \psi_{a,t}\| \| \psi_{a_k,t} \|\leq \|\psi_{a_k,t} - \psi_{a,t}\|.
\end{align*}
Assume $a_k \leq a \leq 2^{\epsilon/2} a_k$, then applying the definition \eqref{eq:defOfHaar} we see that $\|\psi_{a_k,t} - \psi_{a,t}\|^2$ can be estimated by $2 \frac{1}{2a}|a-a_k| + a_k |\frac{1}{\sqrt{2a}}-\frac{1}{\sqrt{2 a_k}}|^2$. Now we have 
$$|a-a_k| \leq a_k (2^{\epsilon/2}-1) < \epsilon/2.$$ 
This yields that 
$$|1-\left \langle \psi_{a,t}, \psi_{a_k,t} \right \rangle |\leq \epsilon^2$$
and hence 
$\delta(\psi_{a,t}, \psi_{a_k,t}) <\epsilon$. 

The case of $2^{\epsilon/2}a_k \leq a \leq a_{k+1}$ follows analogously.
\end{proof}
We denote by
\begin{align*}
 \mathcal{W}(x) = \sup \left\{ |\left \langle \psi_{a,t}, x \right \rangle|,  \psi_{a,t} \in \mathcal{W}) \right \}
\end{align*}
the maximum value of the Haar coefficients of a signal $x$. Similarly
\begin{align*}
 \mathcal{W}_\epsilon(x) = \sup \left\{ |\left \langle \psi_{a,t}, x \right \rangle|,  \psi_{a,t} \in \mathcal{W}_\epsilon) \right \}
\end{align*}
denotes the maximum Haar coefficient of $x$ with respect to the subsampled system $\mathcal{W}_{\epsilon}$.
Using the $\epsilon-$net of Lemma \ref{lem:EpsilonNet} we aim to now construct a reliable test. Indeed we will show, that if the jump size is larger than $\sqrt{2(1+\eta)\log(n)}$ for some positive $\eta$ the \textit{generalized likelihood ratio test} (GLRT), which rejects the null hypothesis if
\begin{align*}
 \mathcal{W}(x) \geq t_{n},
\end{align*}
is asymptotically powerful, if we choose $t_{n}/\sqrt{2 \log (n)} \to 1$.
To prove this, we employ the fact that the subsampled Haar wavelet system yields an $\epsilon$-net. Furthermore, we make use of the $GLRT$ with the reduced dictionary, i.e. we reject the null hypothesis if 
\begin{align}
 \mathcal{W}_\epsilon(x) \geq t_{n}, \label{eq:reducedTest}
\end{align}
where $t_{n}/\sqrt{2\log(n)} \to 1$. The last ingredient for our desired result is the following result describing extreme values of Gaussian processes.
\begin{lemma}\cite{AriDH2005DetectOfGeoObj}\label{lem:MaxOfGaussianVariables}
Let $\omega_1, \dots, \omega_m$ be (possibly dependent) $\mathcal{N}(0,\sigma_i^2)$ variables with all $\sigma_i \leq \sigma$. Then,
\begin{align*}
 P(\max\{\omega_1, \dots, \omega_m\} > \sqrt{2 \log (m)} \sigma \leq \frac{1}{\sqrt{4\pi \log(m)}}.
\end{align*}
\end{lemma}
Now we will prove that asymptotically the reduced GLRT and the full GLRT behave in the same way. For this, let 
$$\delta_\epsilon(x): = \min_{\psi_{a_k,t} \in \mathcal{W_\epsilon}} | \mathcal{W}(x) - \left \langle \psi_{a_k,t}, x \right \rangle|,$$
denote the effective distance of $\mathcal{W}$ to $\mathcal{W}_\epsilon$. 
\begin{lemma}\label{lem:delta}
 For each $\eta>0$,
 $$P_{H_{0,n}}\left\{ \delta_\epsilon(x) > \sqrt{8 \log(n)}\epsilon \right \} \leq \frac{1}{\sqrt{\log(n)}}.$$
\end{lemma}
\begin{proof}
First of all we assume that the maximum of the full GLRT is attained at $\psi_{a,t}$. By Lemma \ref{lem:EpsilonNet} we have that there exists $\psi_{a_k,\tilde{t}} \in \mathcal{W}_\epsilon$ such that $\|\psi_{a,t}- \psi_{a_k,\tilde{t}}\|_2 \leq \epsilon$. Then
 \begin{align*}
\delta(x)  \leq  |\langle \psi_{a_k, \tilde{t}} - \psi_{a,t}, x \rangle|=|\sum_{1\leq i\leq n} (\psi_{a_k,\tilde{t}}(i) - \psi_{a,t}(i)) x(i)|.
 \end{align*}
Since we assume that $H_{0,n}$ holds, we have that $x(i) = z(i) \sim N(0,1)$ for all $1\leq i\leq n$. Now $\sum_i(\psi_{a_k,\tilde{t}}(i) - \psi_{a,t}(i)) z(i) \sim \mathcal{N}(0, \|\psi_{a_k,\tilde{t}}(i) - \psi_{a,t}\|^2) = \mathcal{N}(0, \tilde{\epsilon}^2)$ with $\tilde{\epsilon}\leq \epsilon$.
To estimate the probability of $\delta_\epsilon(x)$ exceeding $4\sqrt{\log(n)}\epsilon$ we have to take the maximum of all probability variables $\sum_i(\psi_{a_k,\tilde{t}}(i) - \psi_{a,t}(i)) z(i)$ over all $a, t$ and corresponding $a_k, \tilde{t}$. These are certainly less than $n^4$ variables which are all normal distributed with variance less than $\epsilon^2$. Hence Lemma \ref{lem:MaxOfGaussianVariables} implies that
\begin{align*}
 P_{H_{0,n}}\left\{ \delta_\epsilon(x) > \sqrt{2\log(n^4)}\epsilon \right \} \leq \frac{1}{\sqrt{4\pi\log(n^4)}} \leq \frac{1}{\sqrt{\log(n)}}.
\end{align*}
\end{proof}
Now we are in the position to prove that the GLRT is asymptotically powerful, should the jump size exceed $\sqrt{2(1+\eta)\log(n)}$.
\begin{theorem}\label{thm:AsymptoticallyPowerfulHaar}
 For each $\eta>0$,
 \begin{align*}
  P_{H_0}\left\{ \mathcal{W}(x) > \sqrt{2(1+\eta) \log(n)}\right \} \to 0\for n \to \infty.
 \end{align*}
\end{theorem}
\begin{proof}
 Let us first prove that, for $\epsilon >0$,
 \begin{align*}
  P_{H_0}\left\{ \mathcal{W}_\epsilon(x) > \sqrt{2(1+\eta) \log(n)}\right \} \to 0\for n \to \infty.
 \end{align*}
In fact $\#\mathcal{W}_\epsilon \leq \frac{1}{\epsilon}n (\log (n)-1)$. Since under $H_0$, $\left \langle \psi_{a_k, t}, x \right \rangle \sim \mathcal{N}(0,1)$ for all $\psi_{a_k, t} \in \mathcal{W}_\epsilon$ Lemma \ref{lem:MaxOfGaussianVariables} implies that
\begin{align*}
 P_{H_0}\left\{ \mathcal{W}_\epsilon(x) > \sqrt{2 \log(\frac{1}{\epsilon}n (\log (n)-1))}\right \} &= P_{H_0}\left\{ \mathcal{W}_\epsilon(x) > \sqrt{2 (\log(n) + \log(\log (n)-1)- \log(\epsilon))}  \right \}\\
 &\leq \frac{1}{\sqrt{4\pi\log(\frac{1}{\epsilon}n (\log (n)-1))}}.
\end{align*}
Furthermore we obtain that for any $\eta>0$, there exists $M_\eta>0$ such that for all $m\geq M_\eta$ we have $\eta\log(m) \geq  \log(\log (m)-1)- \log(\epsilon)$. Hence
\begin{align*}
 P_{H_0}\left\{ \mathcal{W}_\epsilon(x) > \sqrt{2(1+\eta) \log(m)}\right \}&\leq P_{H_0}\left\{ \mathcal{W}_\epsilon(x) > \sqrt{2 (\log(n) + \log(\log (n)-1)- \log(\epsilon))}  \right \} \\
 &\leq \frac{1}{\sqrt{4\pi\log(n)}} \to 0\for n\to \infty.
\end{align*}
The result for $\mathcal{W}(x)$ follows by choosing $\epsilon$ such that $\sqrt{2 (1+\eta)} - \sqrt{8}\epsilon\leq \sqrt{2 (1+\eta/2)}$, then Lemma \ref{lem:delta} and the observation that
\begin{align*}
 \mathcal{W}_\epsilon(x) \leq \mathcal{W}(x) \leq \mathcal{W}_\epsilon(x) + \delta_\epsilon(\xi),
\end{align*}
imply the result.
\end{proof}

Of course the question arises weather the bound of Theorem \ref{thm:AsymptoticallyPowerfulHaar} is tight, i.e. if $\sqrt{2(1+\eta)\log(n)}$ is the minimum jump size that can be detected. The following Lemma states that this is indeed the case, i.e. that a lower jump size will not be detectable with an asymptotically powerful test.
\begin{lemma}\label{lem:theStrongLemma}
Given a sequence of hypotheses
$$H_{1,n}^\eta: x(i) = \xi_J(i) + z(i), \quad 1\leq i \leq n,$$
and $\eta >0$ with
\begin{align*}
 \mathcal{W}_\epsilon(\xi_J) = \sqrt{2(1-\eta)\log(n)}.
\end{align*}
Then, no test is asymptotically powerful for testing $H_{0,n}$ against $H_{1,n}^\eta$. 
\end{lemma}
\begin{proof}

Let $H_{0,n}$ be the hypothesis that $x(i) \sim \mathcal{N}(0,1)$. Furthermore let $H_{1,n}'$ be the hypothesis that for any $0\leq t\leq n/2-1$ 
\begin{align*}
 x(i) = A_n \psi_{1,2t} + z(i), z_{i} \sim \mathcal{N}(0,1)\with 0\leq i\leq n-1.
\end{align*}
Consider the orthonormal Haar wavelet basis on $1\leq i\leq n$ consisting of 
\begin{align*}
 \{\psi_{1,2t}: t\in [0, n/2]\} \cup \left\{\frac{1}{\sqrt{2}} \chi_{[2t,2t+1]}: t\in [0, n/2]\right\}.
\end{align*}
Indeed, since this yields an orthonormal basis for the signals on $I_n$ we can apply a basis transformation and represent $x$ with respect to the Haar wavelet basis to obtain $\hat{x}$. 
Now $H_{1,n}'$ becomes 
\begin{align*}
 \hat{x}(i) = A_n \delta_t(i) + z(i)\with 0\leq i\leq n.
\end{align*}
The needle in a haystack problem, suggests, that if $A_n = \sqrt{2(1-\eta \log(n))}$ for $\eta>0$, then $H_{0,n}$ and $H_{1,n}'$ merge asymptotically, i.e. there does not exist an asymptotically powerful test to distinguish between the two hypotheses. Since distinguishing $H_{1,n}'$ from $H_{0,n}$ should be easier than distinguishing $H_{1,n}$ from $H_{0,n}$ this proves the claim.
\end{proof}
\section{General Geometrical Structures}
\subsection{General Framework} \label{sec:generalFramework}
In this section we will extract the main ideas of the path we followed in the previous subsection on Haar wavelets to give a theory for arbitrary systems of functions. This can then be employed in any dimension and especially in Section \ref{sec:2DJump} to two dimensional functions. 
We saw in the Haar wavelet case, that subsampling the functions dyadically does not destroy the detectability of jumps by that system. 
Indeed, we will show, that whenever we have set of functions on $I_n$: $(g_j)_{j\in J_n}$ with some index set $J_n$ such that for every $\epsilon>0$ there exist subsystems, that yield an \emph{$\epsilon$-net} the reduced GLRT is asymptotically powerful if the signal has a \emph{detectable feature}:
\begin{definition}
Given a function system $(g_j)_{j\in J_n}$ with a subsystem $(g_j)_{j\in J^{\epsilon}_n}, J^{\epsilon}_n \subset J_n$ on $I_n$. If for each $l\in J_n$ we have that $\min_{j\in J_n^\epsilon}\|g_j-g_l\|<\epsilon$, we say that $(g_j)_{j\in J^{\epsilon}_n}$ \emph{constitutes an $\epsilon$-net for $(g_j)_{j\in J_n}$}.
\end{definition}
\begin{definition}\label{def:detectableFeature}
For each $n\in \N$ let $(g_j)_{j\in J_n}$ be a system of functions on $I_n$ with $\|g_j\|_{\ell^2} = \sum_{i\in I_n} |g_j(i)|^2 = 1$. Furthermore assume that for every $\epsilon$ there exists a subsystem $((g_j)_{j\in J_n^\epsilon})$ hat constitutes an $\epsilon$-net for $(g_j)_{j\in J_n}$. Assume that there exists $e>0$ such that for every $\omega>0$ there exists a constant $C_\omega$ with 
\begin{align*}
|J_n^\epsilon| \leq C_\omega n^{e+\omega}\fa n\in N.
\end{align*}A sequence of signals $\xi^{(n)}$ has a \emph{detectable feature} if $\sup_{j\in J_n}|\left \langle \xi^{(n)}, g_j \right \rangle| \geq (1+\eta)\sqrt{2\log(n^e)}$ for some $\eta>0$ and $n\in \N$.
\end{definition}
The jump singularities from the preceding subsection are detectable features of the continuous Haar wavelet system $\mathcal{W}$. We will study another examples in the sequel which are detectable features of continuous shearlet systems.
Moving on with the general set-up and a set of functions $(g_j)_{j\in J_n}$, we define 
\begin{align*}
 \mathcal{G}_n(x): = \max_{j\in J_n} |\left \langle g_j, x \right \rangle|, 
\end{align*}
and for the subsystems we write
\begin{align*}
 \mathcal{G}_n^\epsilon(x): = \max \limits_{j\in J_n^\epsilon} |\left \langle g_j, x \right \rangle|. 
\end{align*}
 $\mathcal{G}_n(x)$ and $\mathcal{G}_n^\epsilon(x)$ denote the maximum response from a signal, when tested with $\mathcal{G}_n$ or $\mathcal{G}_n^\epsilon$ respectively. Using these we can again build a full and a reduced GLRT by checking, whether $\mathcal{G}_n(x)$ or  $\mathcal{G}_n^\epsilon(x)$ exceed a critical threshold. Furthermore, we aim to obtain an analog of Lemma \ref{lem:delta}, which would tell us that asymptotically we can replace the full GLRT with the reduced GLRT.
With 
$$\delta_\epsilon(x): = \min \limits_{j\in J^\epsilon_n} | \mathcal{G}(x) - \left \langle g_j, x \right \rangle|,$$ 
we obtain
\begin{lemma}\label{lem:delta2}
For each $n\in \N$ let $(g_j)_{j\in J_n}$ be a set of functions on $I_n$. Assume that for every $\epsilon>0$ we have subsystems $((g_j)_{j\in J_n^\epsilon)}$ that yield an $\epsilon$-net for $(g_j)_{j\in J_n}$. Then, we have
$$P_{H_{0,n}}\left\{ \delta_\epsilon(x) > \sqrt{8 \log(n)}\epsilon \right \} \leq \frac{1}{\sqrt{\log(n)}}.$$
\end{lemma}
\begin{proof}
 The proof is the same as that of Lemma \ref{lem:delta}.
\end{proof}

Now we can show, that every detectable feature can, as the name already predicts, be detected. Furthermore it is possible to construct a test using only a reduced GLRT.

\begin{theorem}\label{thm:detectWhatCanBeDetected}
For every $n \in \N$ let $(g_j)_{j\in J_n}$ be a set of normalized functions with subsystems as in Definition \ref{def:detectableFeature} that have cardinality $O(n^{e+\omega})$ for $e>0$ and every $\omega >0$. Let $\xi^{(n)}$ be a sequence of signals that have a detectable feature
and 
$$x^n(i) = \xi^{(n)}(i)+ z(i), \text{ where }z(i) \sim \mathcal{N}(0,1) \text{ and } i\in I_n.$$
Then there exists $\epsilon>0$ such that for all $\eta>0$
$$P\left\{\mathcal{G}^\epsilon(z)>(\sqrt{2(1+\eta) \log(n^e)})\right \} \to 0\for n\to \infty$$
and there exists $\eta>0$ such that
$$ P\left\{\mathcal{G}^\epsilon(x^{(n)})>(\sqrt{2(1+\eta) \log(n^e)})\right \} \to 1\for n\to \infty.$$
\end{theorem}
\begin{proof}
The first part of the proof uses similar arguments as the proof of Theorem \ref{thm:AsymptoticallyPowerfulHaar}
For $\epsilon, \eta>0$, there exists $\omega < \eta$ and $C_\omega$ such that $\#\mathcal{G}_\epsilon \leq C_\omega n^{e+\omega}$.
Since the $g_j$ are normalized we obtain that $\left \langle g_j, z \right \rangle \sim \mathcal{N}(0,1)$. Lemma \ref{lem:MaxOfGaussianVariables} implies
\begin{align*}
 P\left\{ \mathcal{W}_\epsilon(z) > \sqrt{2 \log(C_\omega n^{e+\omega})}\right \} &= P\left\{ \mathcal{W}_\epsilon(z) > \sqrt{2 (1+\omega)\log(n^e) - \log(C_\omega))}  \right \}\\
 &\leq \frac{1}{\sqrt{4\pi\log(C_\omega n^{e+\omega})}}.
\end{align*}
Furthermore we obtain that there exists $M_\eta>0$ such that for $n\geq M_\eta$ we have $\eta\log(n^e) \geq  \omega\log(n^e) - \log(C_\omega)$. Hence
\begin{align*}
 P\left\{ \mathcal{W}_\epsilon(z) > \sqrt{2(1+\eta) \log(n^e)}\right \}&\leq P\left\{ \mathcal{W}_\epsilon(z) > \sqrt{2 (1+\omega)\log(n^e) - \log(C_\omega))}  \right \} \\
 &\leq \frac{1}{\sqrt{4\pi\log(n^e)}} \to 0.
\end{align*}
For the second part we first prove that there exists some $\eta>0$ such that
\begin{align}\label{eq:theProbForW}
P\left\{\mathcal{G}(x^{(n)})>(\sqrt{2(1+\eta) \log(n)})\right \} \to 1\for n\to \infty. 
\end{align}
By the assumption that $x$ has a detectable feature we obtain that there exists some $\gamma$ such that $\mathcal{G}(x^{(n)}) \sim \mathcal{N}(\sqrt{2(1+\gamma)\log(n^e)} , 1)$. Choosing $\eta < \gamma$ less than $\gamma$ yields that \eqref{eq:theProbForW} is equal to
\begin{align*}
 P\left\{\mathcal{N}(\sqrt{2(1+\gamma)\log(n^e)} , 1) > \sqrt{2(1+\eta) \log(n)}\right \} =  P\left\{\mathcal{N}(0 , 1) > -\beta \sqrt{\log n} \right \} \leq n^{-\beta^2/2} \to 0,
\end{align*}
where the last equation follows by Mills ratio and $\beta = \sqrt{2(1+\eta)}-\sqrt{2(1+\gamma)}$. To obtain that 
$$P\left\{\mathcal{G}^\epsilon(x^{(n)})>(\sqrt{2(1+\eta) \log(n)})\right \}$$
we observe that
\begin{align*}
\delta_\epsilon(x^{(n)}) &= \min \limits_{j\in J^\epsilon_n} | \mathcal{G}(x^{(n)}) - \left \langle g_j, x^{(n)} \right \rangle| = \min \limits_{j\in J^\epsilon_n} | \mathcal{G}(\xi^{(n)}) - \left \langle g_j, \xi^{(n)} \right \rangle +  \mathcal{G}(z) - \left \langle g_j, z \right \rangle|\\
&\leq \min \limits_{j\in J^\epsilon_n} (| \mathcal{G}(\xi^{(n)}) - \left \langle g_j, \xi^{(n)} \right \rangle |+ | \mathcal{G}(z) - \left \langle g_j, z \right \rangle|).
\end{align*}
Let $g_{\tilde{j}} \in \mathcal{G}$ be a function where $| \langle g_{\tilde{j}}, \xi^{(n)} \rangle| =  \mathcal{G}(\xi^{(n)})$, then
\begin{align*}
 \min \limits_{j\in J^\epsilon_n} (| \mathcal{G}(\xi^{(n)}) - \left \langle g_j, \xi^{(n)} \right \rangle |+ | \mathcal{G}(z) - \left \langle g_j, z \right \rangle|) \leq&  |\left \langle g_{\tilde{j}} -  g_j, \xi^{(n)} \right \rangle| + |\left \langle g_{\tilde{j}} -  g_j, z \right \rangle|\\
 \leq& \epsilon +  |\left \langle g_{\tilde{j}} -  g_j, z \right \rangle|. 
\end{align*}
The rest follows by Lemma \ref{lem:delta2}.
\end{proof}
% \begin{lemma}\label{lem:delta3}
% Let for each $n\in \N$ a function system $(g_j)_{j\in J_n}$ of functions on $I_n$. Furthermore assume that for every $\epsilon>0$ we have subsystems $((g_j)_{j\in J_n^\epsilon)}$ that yield an $\epsilon$-net for $(g_j)_{j\in J_n}$.
% $$P\left\{ \delta_\epsilon(x) > (1+\sqrt{8 \log(n)})\epsilon \right \} \leq \frac{1}{\sqrt{\log(n)}}$$
% \end{lemma}
% \begin{proof}
% We use $x = \xi + z$ to obtain
% \begin{align*}
% \delta_\epsilon(x) &= \min \limits_{j\in J^\epsilon_n} | \mathcal{G}(x) - \left \langle g_j, x \right \rangle| = \min \limits_{j\in J^\epsilon_n} | \mathcal{G}(\xi) - \left \langle g_j, \xi \right \rangle +  \mathcal{G}(z) - \left \langle g_j, z \right \rangle|\\
% &\leq \min \limits_{j\in J^\epsilon_n} (| \mathcal{G}(\xi) - \left \langle g_j, \xi \right \rangle |+ | \mathcal{G}(z) - \left \langle g_j, z \right \rangle|).
% \end{align*}
% Let $g_{\tilde{j}} \in \mathcal{G}$ be the function where the value $\mathcal{G}(\xi)$ is attained, then
% \begin{align*}
%  \min \limits_{j\in J^\epsilon_n} (| \mathcal{G}(\xi) - \left \langle g_j, \xi \right \rangle |+ | \mathcal{G}(z) - \left \langle g_j, z \right \rangle|) \leq  |\left \langle g_{\tilde{j}} -  g_j, \xi \right \rangle| + |\left \langle g_{\tilde{j}} -  g_j, z \right \rangle|\leq \epsilon +  |\left \langle g_{\tilde{j}} -  g_j, z \right \rangle|. 
% \end{align*}
% The rest follows by Lemma \ref{lem:delta2}.
% \end{proof}
Let us reflect shortly on what we proved so far. If a signal has a detectable structure, i.e., a structure that can be detected by a set of functions, we can subsample this set to obtain a cheaper test. We saw analyzed how this can be applied for Haar wavelets. Jump singularities are detectable by the continuous Haar wavelet system as we observed in Subsection \ref{sec:1DJumpSing}. Furthermore the continuous Haar wavelet system can be subsampled to yield $\epsilon$-nets as shown in Lemma \ref{lem:EpsilonNet}. This smaller set allows for the construction of a fast test by computing all the scalar products of the subsampled Haar wavelet coefficients with the signal. In fact this can be done in $O(n\log(n))$, which we see by observing, that $\left \langle \psi_{a_k,t}, x \right \rangle = \psi_{a, 0} * x(t) = \mathrm{ifft}(\mathrm{fft}(\psi_{a, 0}) \mathrm{fft}(x))$, where $\mathrm{fft}$ and $\mathrm{ifft}$ are the discrete fast Fourier transform and its inverse.
Since we have only $O(\log(n))$ scaling parameters, this yields the claimed computational complexity of $O(n\log(n))$. Since this machinery seems quite powerful we can now advance to higher dimensional structures namely those described by continuous shearlet systems.

\subsection{Detecting 2D Jump Singularities}\label{sec:2DJump}
From the work of \cite{KP2014AnaSing} is is clear, that continuous shearlet systems are well versed to detect geometric structures like discontinuities along smooth curves. By virtue of Theorem \ref{thm:detectWhatCanBeDetected} we need to show, that jump singularities are a detectable feature of a continuous shearlet system and that subsampling the shearlet system yields a subsystem, that constitutes an $\epsilon$-net.

Let us start by defining the continuous digital shearlet transform on a $n\times n$ pixel image. As the digital realm should be a discretization of the continuum realm $L^{2}([0,1]^2)$ this should be reflected in the construction.
Let $\psi$ be a continuous shearlet supported in $[0,1]^2$. We define the continuous digital shearlet system as digitization of the continuous system by defining for an $n\times n$ pixel image the continuous digital shearlet system $\mathcal{CDSH}$ as
\begin{align*}
 \mathcal{CDSH}: = \left\{\psi_{a,s,t}, a = 1, \dots n, s = -\frac{n}{2}, \dots, \frac{n}{2}, t= (t_1,t_2), 1\leq t_i \leq n \right \}, 
\end{align*}
where 
\begin{align*}
 \psi_{a,s,t}(i) = n^\frac{5}{4} \int_{x \in Q_{\frac{1}{2n}(i/n)}}a^{-\frac{3}{4}} \psi (A_\frac{a}{n}^{-1} S_\frac{s}{n}^{-1} (x-t/n)) dx.
\end{align*}
In the equation above $Q_{q}(r)$ denotes the cube of sidelength $q$ and center $r$.
As a second step we define subsystems with significantly lower complexity that still yield $\epsilon-$nets.
With $\lceil t \rceil = (\lceil t_1\rceil, \lceil t_2 \rceil)$, we define for $\eta, \omega, \nu>0$,
\begin{align*}
\mathcal{CDSH}^{\delta, \omega, \nu}: = \left\{\psi_{\lceil a_k \rceil ,\lceil s \rceil,\lceil t \rceil},\ a_k = 2^{\delta k},\ s = \omega \sqrt{a_k n}  l,\ t = (t_1,t_2),\ t_1 = \nu a_k r_1,\ t_2 = \nu \sqrt{a_k n} r_2, \right.\\
k=0, \dots, \left\lceil\frac{1}{\delta}(\log(n)-1)\right\rceil,\ l = -\left\lfloor \frac{\sqrt{n}}{\sqrt{a_k}\omega}\right\rfloor, \dots, \left\lfloor \frac{\sqrt{n}}{\sqrt{a_k}\omega}\right\rfloor,\\
\left. r_1 = 0,\dots,\frac{n}{a_k \nu},\  ,\ r_2 = 0, \dots, \frac{\sqrt{n}}{\sqrt{a_k} \nu}\right\}.
\end{align*}
Observe that the cardinality of $\mathcal{CDSH}^{\delta, \omega, \nu}$ is of order $n^{2}(\log n)$, which compared to the colossal $O(n^4)$ of the $\mathcal{CDSH}$ can be seen as a significant reduction. In order to apply Theorem \ref{thm:detectWhatCanBeDetected} we still need to see that the new system constitutes an $\epsilon$-net.
\begin{lemma}\label{lem:ShearletsEpsilonNets}
Let $\psi$ be a Lipschitz continuous shearlet with Lipschitz constant $C_\psi$. Then for $\delta \leq 1$, $\mathcal{CDSH}^{\delta, \omega, \nu}$ is an $\epsilon$-net for $\mathcal{CDSH}$ with $\epsilon = 2\delta/3 + C_\psi  3\delta + \omega  + 3C_\psi\nu$.
\end{lemma}
\begin{proof}
 We compute
 \begin{align*}
  \|\psi_{a,s,t} - \psi_{a_k,\tilde{s},\tilde{t}}\|,
 \end{align*}
where we need to choose $k$, $\tilde{s}$ and $\tilde{t}$ in a suitable way. Without loss of generality, we can assume $t = 0$, since the above expression is invariant under translations. Furthermore, we choose 
\begin{align*}
 a_k = \argmin_k |\lceil a_k \rceil - a|,\quad \tilde{s} = \argmin \limits_{\hat{s}=-\frac{n}{2}:\omega \sqrt{a_k n}:\frac{n}{2}} |s-\lceil\hat{s}\rceil|.
\end{align*}
Assume for simplicity that $a_k<a$. Then we obtain the error estimates
\begin{align*}
  |\lceil a_k \rceil - a| \leq a_k (2^{\delta}-1) ,\quad |s-\lceil\tilde{s}\rceil| \leq \omega \sqrt{a_k n}/2.
\end{align*}
Moreover,choose $\tilde{t} = (\tilde{t}_1, \tilde{t}_2)$ such that
\begin{align*}
 |\tilde{t}_1| \leq  \frac{\nu}{2} a_k, \quad |\tilde{t}_2| \leq \frac{\nu}{2}\sqrt{a_k n}.
\end{align*}
Then we obtain 
\begin{align*}
 &\|\psi_{a,s,0} - \psi_{a_k,\tilde{s},\tilde{t}}\|^2\\
 &\quad= \sum \limits_{i\in I_n}\left| \psi_{a,s,0}(i) - \psi_{a_k,\tilde{s},\tilde{t}}(i)\right|^2\\
 &\quad=\sum \limits_{i\in I_n}\left|n^\frac{5}{4} \int_{x \in Q_{\frac{1}{2n}(i/n)}}a^{-\frac{3}{4}} \psi (A_\frac{a}{n}^{-1} S_\frac{s}{n}^{-1} x) - \lceil a_k\rceil^{\frac{3}{4}}\psi (A_\frac{\lceil a_k\rceil}{n}^{-1} S_\frac{\lceil \tilde{s}\rceil}{n}^{-1} (x-\tilde{t}/n))dx \right|^2\\
  &\quad=\sum \limits_{i\in I_n}\left|n^\frac{5}{4} \int_{x \in Q_{\frac{1}{2n}(i/n)}}a^{-\frac{3}{4}} \psi (A_\frac{a}{n}^{-1} S_\frac{s}{n}^{-1} x) - \lceil a_k\rceil^{-\frac{3}{4}}\psi (A_\frac{\lceil a_k\rceil}{n}^{-1} S_\frac{\lceil \tilde{s}\rceil}{n}^{-1} x-A_\frac{\lceil a_k\rceil}{n}^{-1} S_\frac{\lceil \tilde{s}\rceil}{n}^{-1}\tilde{t}/n) dx \right|^2.
\end{align*}
Now we apply the transformation $x \mapsto S_\frac{\lceil \tilde{s}\rceil}{n} A_\frac{\lceil a_k\rceil}{n}x$ to the last term above. We obtain that
\begin{align*}
\|\psi_{a,s,0} - \psi_{a_k,\tilde{s},\tilde{t}}\|^2 = \sum \limits_{i\in I_n}|n^{-\frac{1}{2}} \int_{D_i} \psi( x ) - \left(\frac{a}{\lceil a_k\rceil}\right)^{\frac{3}{4}}\psi (A_\frac{\lceil a_k\rceil}{n}^{-1} S_{\frac{s-\lceil \tilde{s}\rceil}{n}} A_\frac{a}{n}x - A_\frac{\lceil a_k\rceil}{n}^{-1} S_\frac{\lceil \tilde{s}\rceil}{n}^{-1}\tilde{t}/n) dx |^2,
\end{align*}
where $D_i = \left(S_\frac{\lceil \tilde{s}\rceil}{n} A_\frac{\lceil a_k\rceil}{n}\right)^{-1} Q_{\frac{1}{2n}(i/n)}$.
Since $|D_i|  \leq \frac{1}{\sqrt{n}}$ we only need to show that for $x\in D_i$
\begin{align*}
 |\psi( x ) - \left(\frac{a}{\lceil a_k\rceil}\right)^{\frac{3}{4}}\psi (A_\frac{\lceil a_k\rceil}{n}^{-1} S_{\frac{s-\lceil \tilde{s}\rceil}{n}} A_\frac{a}{n}x - A_\frac{\lceil a_k\rceil}{n}^{-1} S_\frac{\lceil \tilde{s}\rceil}{n}^{-1}\tilde{t}/n) |^2 < 2\delta/3 + C_\psi  3\delta + \omega  + 3C_\psi\nu.
\end{align*}
We estimate as follows:
\begin{align*}
 &\left| \psi( x ) - \left(\frac{a}{\lceil a_k\rceil}\right)^{\frac{3}{4}}\psi (A_\frac{\lceil a_k\rceil}{n}^{-1} S_{\frac{s-\lceil \tilde{s}\rceil}{n}} A_\frac{a}{n}x - A_\frac{\lceil a_k\rceil}{n}^{-1} S_\frac{\lceil \tilde{s}\rceil}{n}^{-1}\tilde{t}/n) \right|\leq  I_1 + I_2 + I_3 , 
\end{align*}
where
\begin{align*}
 I_1:& = \left| \psi( x ) - \left(\frac{a}{\lceil a_k\rceil}\right)^{\frac{3}{4}}\psi( x ) \right|\\
 I_2:& = \left(\frac{a}{\lceil a_k\rceil}\right)^{\frac{3}{4}} \left| \psi( x )- \psi (A_\frac{\lceil a_k\rceil}{n}^{-1} S_{\frac{s-\lceil \tilde{s}\rceil}{n}} A_\frac{a}{n}x ) \right|\\
 I_3:&= \left(\frac{a}{\lceil a_k\rceil}\right)^{\frac{3}{4}} \left| \psi( A_\frac{\lceil a_k\rceil}{n}^{-1} S_{\frac{s-\lceil \tilde{s}\rceil}{n}} A_\frac{a}{n}x) - \psi (A_\frac{\lceil a_k\rceil}{n}^{-1} S_{\frac{s-\lceil \tilde{s}\rceil}{n}} A_\frac{a}{n}x - A_\frac{\lceil a_k\rceil}{n}^{-1} S_\frac{\lceil \tilde{s}\rceil}{n}^{-1}\tilde{t}/n )\right|. 
\end{align*}
Clearly $I_1 \leq (2^{\frac{3}{4}\delta}-1) \leq 2\delta/3$. The cases $I_2, I_3$ follow by the Lipschitz continuity of $\psi$ by
\begin{align*}
 I_2 &\leq 2^{\frac{3}{4}\delta} C_\psi | x -  A_\frac{\lceil a_k\rceil}{n}^{-1} S_{\frac{s-\lceil \tilde{s}\rceil}{n}} A_\frac{a}{n}x| \text{ and }I_3 &\leq 2^{\frac{3}{4}\delta} C_\psi | A_\frac{\lceil a_k\rceil}{n}^{-1} S_\frac{\lceil \tilde{s}\rceil}{n}^{-1}\tilde{t}/n|.
\end{align*}
First of all
\begin{align*}
 A_\frac{\lceil a_k\rceil}{n}^{-1} S_{\frac{s-\lceil \tilde{s}\rceil}{n}} A_\frac{a}{n} &=  \begin{pmatrix}
                                                                                         \frac{n}{\lceil a_k\rceil} & 0 \\
                                                                                          0& \sqrt{\frac{n}{\lceil a_k\rceil}}\\
                                                                                          \end{pmatrix}
                                                                                          \begin{pmatrix}
                                                                                          1 & \frac{s-\lceil \tilde{s}\rceil}{n} \\
                                                                                          0& 1\\
                                                                                          \end{pmatrix}   
                                                                                          \begin{pmatrix}
                                                                                          \frac{a}{n} & 0 \\
                                                                                          0& \sqrt{\frac{a}{n}}\\
                                                                                          \end{pmatrix}\\
                                                                                          &=
                                                                                          \begin{pmatrix}
                                                                                          \frac{n}{\lceil a_k\rceil} &  \frac{s-\lceil \tilde{s}\rceil}{\lceil a_k\rceil}\\
                                                                                          0& \sqrt{\frac{n}{\lceil a_k\rceil}}\\
                                                                                          \end{pmatrix}
                                                                                          \begin{pmatrix}
                                                                                          \frac{a}{n} & 0 \\
                                                                                          0& \sqrt{\frac{a}{n}}\\
                                                                                          \end{pmatrix}\\
                                                                                          &=
                                                                                          \begin{pmatrix}
                                                                                          \frac{a}{\lceil a_k\rceil} &  \frac{(s-\lceil \tilde{s}\rceil )\sqrt{a}  }{\lceil a_k\rceil \sqrt{n}} \\
                                                                                          0& \sqrt{\frac{a}{\lceil a_k\rceil}}\\
                                                                                          \end{pmatrix}.\\
\end{align*}
This implies that $I_2 \leq 2^{\frac{3}{4}\delta} C_\psi (|x - \frac{a}{\lceil a_k\rceil} x| + \frac{(s-\lceil \tilde{s}\rceil )\sqrt{a}  }{\lceil a_k\rceil \sqrt{n}}|x_2| )$.
Using that $\frac{a}{\lceil a_k\rceil}  = 2^\delta$ and $\frac{(s-\lceil \tilde{s}\rceil )\sqrt{a}  }{\lceil a_k\rceil \sqrt{n}} \leq \omega/2$ we obtain that $I_2 \leq C_\psi 2^{\frac{3}{4}\delta}((2^{\delta}-1) + \omega/2 )\|x\|_2$ and since $\|x\|_2 \leq \sqrt{2}$ we get $I_2 \leq C_\psi  2^{\frac{3}{4}\delta+\frac{1}{2}}((2^{\delta}-1) +\omega/2 ) \leq C_\psi  3\delta + \omega$.

With a similar approach we can estimate $I_3$ by 
\begin{align*}
2^{\frac{3}{4}\delta} C_\psi \begin{pmatrix}
   \frac{n}{\lceil a_k\rceil} & -\frac{\lceil \tilde{s}\rceil}{\lceil a_k\rceil}\\
   0 & \sqrt{\frac{n}{\lceil a_k\rceil}}\\
  \end{pmatrix}\tilde{t}/n &=2^{\frac{3}{4}\delta} C_\psi \begin{pmatrix}
   \frac{1}{\lceil a_k\rceil} & -\frac{\lceil \tilde{s}\rceil}{n \lceil a_k\rceil}\\
   0 & \frac{1}{\sqrt{n \lceil a_k\rceil}}\\
  \end{pmatrix}\tilde{t}\\
  &\leq 2^{\frac{3}{4}\delta} C_\psi\left( \frac{a_k}{\lceil a_k\rceil} \frac{\nu}{2} + \frac{\sqrt{a_k} \lceil \tilde{s}\rceil}{\sqrt{n} \lceil a_k\rceil }\frac{\nu}{2} + \frac{\sqrt{a_k}}{\sqrt{\lceil a_k\rceil }} \frac{\nu}{2}\right) \leq 3C_\psi\nu.
\end{align*}
In conclusion we obtain that 
$$ I_1 + I_2 + I_3 \leq 2\delta/3 + C_\psi  3\delta + \omega  + 3C_\psi\nu$$
which yields the claim.
\end{proof}
Next we can combine Theorem \ref{thm:detectWhatCanBeDetected} and Lemma \ref{lem:ShearletsEpsilonNets} to obtain, that subsampled shearlet systems can detect all detectable features of unsubsampled shearlet systems. 
\begin{theorem}\label{thm:mainTheorem}
 Let $\psi$ be a Lipschitz continuous shearlet and $\mathcal{CDSH}$ the corresponding continuous digital shearlet system. Let $\xi^{(n)}$ be a sequence of signals that have a detectable feature.
 Then, there exists $\delta, \omega, \nu>0$ such that the GLRT with $\mathcal{CDSH}^{\delta, \omega, \nu}$ is asymptotically powerful.
\end{theorem}

\subsection{Open Questions}
\noindent
\emph{What can be detected by $\mathcal{CDSH}$?}

\noindent
First of all it is not yet clear, what type of geometrical structures can be detected with the $\mathcal{CDSH}$. Certainly, elements of $\mathcal{CDSH}$ are detectable with the GLRT and a reduced shearlet system $\mathcal{CDSH}^{\delta,\omega, \nu}$, but one would certainly be interested in objects or structures that appear naturally in images. Since the publications \cite{KLCharacterizationOfEdges2009, GLLEdgeAnalysis2009, KP2014AnaSing} show how shearlets can detect jump singularities, an extension to these structures seems very likely.

\medskip
\noindent
\emph{Is the subsampling optimal?}

\noindent
We see that the cardinality of the subsampled shearlet systems $\mathcal{CDSH}^{\delta, \omega, \nu}$ is $O(n^2 \log(n))$. Given that the image has $n^2$ pixels, it seems very likely that, at least up to the $\log$ factor, this is the minimal cardinality for which an $\epsilon$-net can be constructed. A proof for this optimality is still missing.  
A different approach would be to allow further subsampling on the cost of an increasing signal strength. If this is possible, a precise characterization of the trade-off as in \cite{nadlerDetectionOfLines} is desirable.

\medskip \noindent
\emph{Extensions}

\noindent
Similar to the question of what can be detected, we can also adjust the process of testing. For instance if we assume we have an image composed out of only a few elements of $\mathcal{CDSH}$, we can develop a procedure, where we find some promising elements in $\mathcal{CDSH}^{\delta, \omega, \nu}$ and try various combinations of them. A similar method to extract curves from noisy images based on extensions of promising parts of curves has been proposed in \cite{AlpNBDetectingFaintCurves}.
Since shearlets yield optimally sparse approximations of cartoon-like images \cite{KLcmptShearSparse2011}, this might ultimately give rise to a method to optimally detect cartoon-like images embedded in noise.

\small
\bibliographystyle{abbrv} 
\bibliography{references}
\end{document}